\documentclass{amsart}

\usepackage{amsthm,amssymb,verbatim,url}
\usepackage{graphicx}
\usepackage{enumerate}

\usepackage[alphabetic]{amsrefs}

\newcommand{\interior}{\operatorname{interior}}

\newcommand{\diam}{\operatorname{diam}}

\newcommand{\Lip}{\operatorname{Lip}}

\newcommand{\sing}{\operatorname{sing}}

\newcommand{\HH}{\mathcal H} 
\newcommand{\LL}{\mathcal L} 

\newcommand{\vv}{\bold v}

 \newcommand{\Cc}{\mathcal C}

 \newcommand{\FF}{\mathcal{F}}
 \newcommand{\RR}{\mathbf{R}}  
 \newcommand{\BB}{\mathbf{B}}  
 \newcommand{\Ss}{\mathbf{S}}  
 \newcommand{\dist}{\operatorname{dist}}

 \newcommand{\eps}{\epsilon}

 \newcommand{\vol}{\operatorname{vol}}

\swapnumbers  

    \newtheorem{theorem}    {Theorem}       [section]
    \newtheorem{lemma}      [theorem]       {Lemma}
    \newtheorem{corollary}  [theorem]     {Corollary}
    \newtheorem{proposition}       [theorem]       {Proposition}

    \theoremstyle{definition}
    \newtheorem{definition}  [theorem] {Definition}
    \theoremstyle{definition}
    \newtheorem{remark}   [theorem]       {Remark}
    \newtheorem*{remark*} {Remark}

\usepackage[pagebackref]{hyperref}
\usepackage[alphabetic]{amsrefs}

\begin{document}

  \newcommand{\marginpor}[1]{}

\renewcommand{\thesubsection}{\thetheorem}

\title[Topological Change in Mean Convex Mean Curvature Flow]{Topological Change in Mean Convex \\ Mean Curvature Flow}
\author{Brian White}
\thanks{This research was supported by the National Science Foundation 
  under grants~DMS-0406209 and~DMS-1105330.}
\email{white@math.stanford.edu}
\date{July 30, 2011.  Revised March 26, 2012}
\begin{abstract} Consider the mean curvature flow of an $(n+1)$-dimensional compact, mean convex region
in Euclidean space (or, if $n<7$, in a Riemannian manifold).   We prove that elements of
the $m^{\rm th}$ 
homotopy group of the complementary region 
can die only if there is a shrinking $\Ss^k\times \RR^{n-k}$
singularity for some $k\le m$.  We also prove that for each $m$ with $1\le m\le n$,
there is a nonempty open set of compact, mean convex regions $K$ in $\RR^{n+1}$ with smooth boundary 
$\partial K$ for which the resulting mean curvature flow has a shrinking $\Ss^m\times\RR^{n-m}$ singularity.
\end{abstract}

\maketitle

\section{Introduction}

Let $K(t)$ be a compact, time-dependent  
region in a Riemannian manifold such that the boundary $\partial K(t)$ moves by mean curvature flow.
Clearly the topology of the complement $K(t)^c$ can change only if there is a singularity of the flow.   It is natural to ask
if we can deduce properties of the singularities from the way in which the topology changes.   In this paper,
we give a rather precise answer if the regions are mean convex.
In particular, 
consider a mean curvature flow $t\in [0,\infty)\mapsto K(t)$ of compact
regions in an $(n+1)$-dimensional Riemannian manifold $N$ such that $K(0)$ is mean
convex and has smooth boundary.  If $n\ge 7$, we require that 
 $N$ be $\RR^{n+1}$ with the 
 Euclidean metric.\footnote{None of the arguments in this paper depend on dimension.
However, they do require that the singularities of the flow 
have convex type (as defined in \S\ref{prelim}), 
and in high dimensions it has not been proved that all singularities have
convex type
except when the ambient space is Euclidean.}.
We prove a theorem that implies the following:

\begin{theorem}\label{IntroTheorem}
Suppose that $0\le a < b$ and that
 there is a map of the $m$-sphere into $K(a)^c$ that is homotopically trivial in $K(b)^c$
but not in $K(a)^c$. 

Then at some time $t$ with $a\le t < b$, there is a singularity of the flow
at which the Gaussian density is $\ge d_m$, the Gaussian density of a shrinking $m$-sphere
in $\RR^{m+1}$, and at which the tangent flow is a shrinking 
      $\Ss^k\times\RR^{n-k}$ for some $k$ with $1\le k\le m$.
\end{theorem}

The following is an interesting special case:

\begin{corollary}\label{corollary}
Suppose that $K$ is a compact, mean convex subset of $\RR^{n+1}$ with smooth boundary,
and suppose that there is a map of the $m$-sphere into $K^c$ that is homotopically nontrivial.
Then the resulting mean curvature
flow has a singularity at which the Gaussian density $\ge d_m$, and at which the tangent flow
is a shrinking $\Ss^k\times \RR^{n-k}$ for some $k$ with $1\le k\le m$.
\end{corollary}

The corollary follows from the theorem because compact subsets of Euclidean space
disappear in finite time under mean curvature flow. (Thus we can choose $b$ large enough that $K(b)$ is empty.)

More generally, the topological assumption in Theorem~\ref{IntroTheorem} can
be replaced by the weaker assumption that
there is a continuous map
\[
   \text{ $F: \BB^{m+1} \to K(b)^c$ with $F(\partial \BB^{m+1})\subset K(a)^c$}
\]
that cannot be homotoped by a $1$-parameter family of such maps to a map
$F'$ whose image $F'(\BB^{m+1})$ lies in $K(a)^c$.  See Theorem~\ref{main}.  
(The $m$ in Theorem~\ref{main} corresponds to $(m+1)$ here.)

In Theorems~\ref{IntroTheorem} and~\ref{main}, the moving hypersurfaces $\partial K(t)$ have
no boundary.  Those theorems generalize to hypersurfaces with boundary, where
the motion of the boundary is prescribed:  see Theorem~\ref{main-boundary}.  The theorems
also generalize to flows in which the surfaces move with normal velocity equal to 
the mean curvature plus the normal component
of a smooth vectorfield on the ambient space. See the discussion in \S\ref{abstract-form-section}.

The Gaussian density inequalities in the various theorems are sharp.
For example, to see that the inequality in Corollary~\ref{corollary} is sharp  
in the case $n=3$ and $m=1$,  
let $K$ be a thin, rotationally symmetric, solid torus in $\RR^3$,
i.e., the set of points at distance $\le \eps$ from a round circle in $\RR^3$.  If $\eps$ is sufficiently small,
$K$ will be mean convex and its complement will have nontrivial fundamental group. 
Furthermore, using the rotational symmetry, it is not hard to show that (under the flow) the surface
collapses to a round circle, that each singularity is a shrinking $\Ss^1\times\RR^1$, and therefore that
the Gaussian density is $d_1$.
To see the sharpness for arbitrary dimensions $n$ and $m$, 
 let $K$ be the set of
points in $\RR^{n+1}$ at distance $\le\eps$ from a round $(n-m)$-sphere in $\RR^{n+1}$.
Then $K^c$ has nontrivial $m^{\rm th}$ homotopy and the singularities all have Gaussian
density $d_m$.

The reader may wonder whether the density inequalities in the various theorems could be replaced
by suitable equalities.   For example, could Theorem~\ref{IntroTheorem} be strengthened
to say that there is a singularity of Gaussian density $d_k$ with $k$ equal to the smallest
$m$ for which the theorem's hypotheses hold?
The answer is no: even in the case of a single spacetime singularity, the change in topology does
not determine the singularity type.   For example,
Altschuler, Angenent, and Giga~\cite{AAG} proved that that there is a ``doubly-degenerate neckpinch'' 
mean curvature flow in $\RR^3$:
the moving surface is a smooth, axially symmetric, mean convex topological sphere until
it collapses to a point, at which point the tangent flow is not a shrinking $\Ss^2$ but rather
a shrinking $\Ss^1\times \RR$.   Compare this flow to a flow consisting of shrinking round
sphere in $\RR^3$.
The two flows are  completely equivalent topologically, yet the singularities are different.
In general, for every $k$ and $n$ with $1\le k <$n, one can construct 
a mean convex $n$-sphere in $\RR^{n+1}$ that shrinks to a point but whose tangent flow
at that point is a shrinking $\Ss^k\times\RR^{n-k}$.

The theorems mentioned so far are about changes in topology of the complements $K(t)^c$
of the moving regions $K(t)$.  One can also ask about changes in topology of the
regions themselves.   By standard topological duality theorems, 
the results described above imply results about the regions $K(t)$.
See Section~\ref{duality-section}.   

In Section~\ref{persist-section}, the results described above, in particular Theorem~\ref{IntroTheorem} 
and Corollary~\ref{corollary}, are used to prove a result about persistence of singularities
under perturbations of the initial surface: we prove that for $1\le m\le n$, there is a nonempty
open set of compact, mean convex regions in $\RR^{n+1}$ with smooth boundary such
that the resulting mean curvature flow has a shrinking $\Ss^m\times \RR^{n-m}$ singularity.
To put this in context, let $\mathcal{C}_n$ be the collection of tangent flows
at singularities of mean curvature flows of hypersurfaces in $\RR^{n+1}$, where two tangent flows are identified
if they are related by an ambient isometry.
We can think of elements of $\mathcal{C}_n$ as singularity types.   Let us call a singularity type $T$
in $\mathcal{C}_n$ {\em avoidable} if the following holds: for a 
  generic smooth, compact, emdedded hypersurface in $\RR^{n+1}$, $T$ does not occur
  as a tangent flow in 
  the  resulting mean curvature 
  flow\footnote{\,In general, one does not have uniqueness (after the first singularity)
  for mean curvature flow: there may be more than one mean curvature flow with a given 
  initial surface.  
  However, in the definition of avoidable singularity types, only generic initial surfaces matter.
  For generic initial surfaces, the flow is unique and so one may speak of {\em the} resulting
  mean curvature flow.}.
  Otherwise, we call
  $T$ {\em non-avoidable}.
Thought experimentation suggests that the non-avoidable singularity types
are precisely the shrinking spheres and cylinders, i.e., the $\Ss^m\times \RR^{n-m}$.
   Theorem~\ref{persist-theorem}
shows that each $\Ss^m\times \RR^{n-m}$ is indeed non-avoidable.
In the other direction, Colding and Minicozzi~\cite{CM-Generic-I} have made
important progress
toward proving that all other singularity types are avoidable. 
The most basic open question is whether $n$-planes of multiplicity $>1$ are avoidable.
(It seems likely that an $n$-plane of multiplicity $>1$ is not merely avoidable, but that in fact
it can never occur as a tangent flow if the initial
surface is compact and smoothly embedded.)

The theorems in this paper all assume mean convexity.  The reader may wonder
what happens if that assumption is dropped.
Certain analogs of the theorems hold for arbitrary (i.e., not  necessarily mean convex)
hypersurfaces in $\RR^3$ and $\RR^4$; see~\cite{Ilmanen-White-GaussDensity}.
For hypersurfaces (of any dimension) that do not fatten under level set flow, there are some restrictions on
the way that the topology of the complement can change, no matter what kinds of singularities 
occur.  See~\cite{White-Topology}.

In~\cite{Ilmanen-White-GaussDensity},  the results of this paper are used to get lower Gaussian 
density bounds
on self-similar shrinkers for mean curvature flow, and
in~\cite{Ilmanen-White-MinimizingCone} and \cite{Ilmanen-White-MinimalCone} they are used to get lower bounds for
densities of minimal cones. 

The results of this paper rely strongly on properties of singularities of mean convex mean curvature
flow that were proved in~\cite{White-Size} and~\cite{White-Nature}.   
However, the density bounds in this paper are vacuously true for flows that have any singularities
with Gaussian density $\ge 2$.  Thus for the density bounds in the paper, one only needs the results
of~\cite{White-Size} and~\cite{White-Nature} under the assumption that the singularities have Gaussian density $<2$, and  
the most complicated parts of those papers are trivially true under that assumption.

\section{Preliminaries}\label{prelim}

In this section, we state the facts about mean curvature flow of mean convex sets that
are important for this paper.  Let $t\in [0,\infty)\mapsto K(t)$ be a mean curvature flow of mean convex subsets
of a smooth,  $(n+1)$-dimensional Riemannian manifold.  If $x\in\partial K(t)$ is a regular
point, we let $\kappa_1(x)\le \kappa_2(x)\le \dots \le \kappa_n(x)$ be the principal curvatures
of $\partial K(t)$ with respect to the inward unit normal normal.   (We could also write $\kappa_i(x,t)$, but
since the surfaces $\partial K(t)$ for distinct values of $t$ are disjoint, $t$ is determined by $x$.)
We let $h(x)=\kappa_1(x) +\dots +\kappa_n(x)>0$ be the scalar mean curvature.
We say that a singular point $x\in \partial K(t)$ (where $t>0$) has {\em convex type} provided
\begin{enumerate}
\item\label{cylinder} Each tangent flow at $x$ is a self-similarly shrinking $\Ss^k\times \RR^{n-k}$
  for some $k\ge 1$.
\item\label{umbilicity} If $x_i\in \partial K(t_i)$ is a sequence of regular points converging to $x$, then
\[
  \liminf\frac{\kappa_1(x_i)}{h(x_i)} \ge 0.
\]
\end{enumerate}
(Actually \eqref{cylinder} follows from~\eqref{umbilicity}, which in turn follows from the seemingly weaker
assumption that the $\liminf$ in~\eqref{umbilicity} is $>-\infty$ for every such sequence $x_i$.
See~\cite{White-Nature}.)

In many situations, singularities are known to have convex type:

\begin{proposition}\label{ConvexType}
Suppose that $K$ is a compact, mean convex region in an $(n+1)$-dimensional Riemannian
manifold. Let $t\in [0,\infty) \mapsto K(t)$ be the mean curvature flow with $K(0)=K$.  Suppose that
\begin{enumerate}[\upshape (1)]
\item\label{n<7} $n<7$, or
\item\label{flat} $\partial K$ is smooth and $N=\RR^{n+1}$.
\end{enumerate}
Then for $t>0$, the singularities of the flow all have convex type.
\end{proposition}
See \cite{White-Nature} for the proof in the case~\eqref{n<7} and~\cite{White-Subsequent}
for the proof in case~\eqref{flat}.

\begin{proposition}\label{NearestPoint}
Let $t\in [a,b]\mapsto K(t)$ be a mean curvature flow of mean convex sets, and suppose that
the singularities of the flow have convex type.
Let $t\in (a,b]$ and let $x$ be a point in the interior of $K(t)$.  Let $y$ be a point in $\partial K(t)$
that minimizes distance to $x$.  Then $y$ is a regular point of the flow.
\end{proposition}

\begin{proof}
Note that $K(t)$ contains the ball with center $x$ and radius $\dist(x,y)$, from which it follows
\[
  \liminf_{r\to 0} \frac{\vol(K(t)\cap \BB(y,r))}{\vol(\BB(y,r))} \ge \frac12.
\]
On the other hand, if $z\in \partial K(t)$ is a singular point of convex type, then
it is straightforward to show that
\[
    \lim_{r\to 0} \frac{\vol(K(t)\cap \BB(z,r))}{\vol(\BB(z,r))} = 0.
\]
\end{proof}

\begin{proposition}[Stone]\label{Stone}
Let $x$ be a convex-type singularity of a mean convex mean curvature flow.
Then there is a $k=k(x)\ge 1$ such that
every tangent flow at $x$ is a shrinking $\Ss^k\times \RR^{n-k}$, where $k$ depends only 
on the Gaussian density $\Theta$ at the point $x$. (It does not depend on the sequence of 
spacetime dilations
used to obtain the tangent flow.)
\end{proposition}

Thus the tangent flow is unique up to rotations. 
For the reader's convenience, we give the idea of Stone's proof.  See \cite{Stone-Density}*{Appendix A}
for details.

\begin{proof}
The Gaussian density $d_k$ of a shrinking $\Ss^k\times \RR^{n-k}$
(where $\Ss^k$ is the unit $k$-sphere in $\RR^{k+1}$)
may be calculated explicitly:
\begin{align*}
    d_k 
    &= \left(\frac{k}{2\pi e}\right)^{k/2}\sigma_k  \\
    &=  \left( \frac{k}{2e}\right)^{k/2} \left( \frac{2\sqrt{\pi}}{\Gamma(\frac{k+1}2)} \right),
\end{align*}
where $\sigma_k$ is the area of a $k$-dimensional sphere of radius $1$.
(As the notation indicates, the value of $d_k$ turns out not to depend on $n$.)
Using this formula, one can show that
\begin{equation}\label{DensitiesDecrease}
   d_1>d_2> \dots.
\end{equation}
Now if a shrinking $\Ss^k\times \RR^{n-k}$ is a tangent flow to $t\mapsto K(t)$
at the point $(x,t)$, then $d_k=\Theta$.  Thus by~\eqref{DensitiesDecrease},
 $k$ is determined by $\Theta$.
 \end{proof}
 
\begin{proposition}\label{SingularSetSize}
Let $\Sigma_k$ be the set of spacetime points $(x,t)$ such that
\begin{enumerate}
\item $t>0$,
\item  $x\in \partial K(t)$,
\item $x$ is a singular point of convex type, and 
\item the Gaussian density at $(x,t)$ is $d_k$ (or, equivalently, the tangent flows at
$x$ are shrinking $\Ss^k\times \RR^{n-k}$s.)
\end{enumerate}
Then $\Sigma_k$ has parabolic Hausdorff dimension at most $(n-k)$.
\end{proposition}

This follows easily from standard dimension reducing.  (It also a special case
of the stratification theory in~\cite{White-Stratification}*{\S9}.)   
Actually, in this paper, we do not
need the full strength of Proposition~\ref{SingularSetSize}.  
 All we need is the following much weaker corollary:

\begin{corollary}\label{SingularSetSizeCorollary}
Suppose that at a certain time $t>0$, the singularities all have convex type with Gaussian density
 $\le d_k$.  Then $K(t)$ is a smooth $(n+1)$-manifold with boundary except for a closed
subset of $\partial K(t)$ whose Hausdorff dimension is at most $(n-k)$.
\end{corollary}

\begin{proposition}\label{MainBlowUpProposition}
Let $t>0$ and let $p\in \partial K(t)$ be a either a regular point or a convex-type singular point at which the
Gaussian density $\Theta(p)$ of the flow is $\le d_m$.   Let $x_i$ be a sequence of points in 
the interior of $K(t)$ that converge to $p$.   Let $y_i$ be a point in $\partial K(t)$
that minimizes distance to $x_i$.   Translate $K(t)$ by $-y_i$ and dilate by
$1/\dist(x_i,y_i)$ to get a set $K_i$.   Then a subsequence $K_{i(j)}$ converges to a convex
set $K'$ with smooth boundary, and the convergence $\partial K_{i(j)}\to \partial K'$ is smooth
on bounded sets.

Furthermore, the homotopy groups $\pi_j(\partial K')$ are trivial for $j< m$.
\end{proposition}


\begin{proof}
The assertion is trivially true if $p$ is a regular point (in that case, the set $K'$ is a closed halfspace),
so we assume that $p$ is a singular point.
Except for the assertion about homotopy groups, this is proved in~\cite{White-Nature}, which
also shows that, after a rotation, either
\begin{enumerate}[\upshape (i)]
\item\label{graph} $\partial K$ is the graph of an entire function from $\RR^n$ to $\RR$, or
\item\label{solid-cylinder} $K$ has the form $C\times \RR^{n-k}$ for some $k\ge 1$, where $C$ 
is a compact, convex subset of $\RR^{k+1}$.
\end{enumerate}
In the first case, all the homotopy groups of $\partial K$ are trivial.
Thus we may assume that $K$ has the form $C\times \RR^{n-k}$ as in~\eqref{solid-cylinder}.

Recall that the {\em entropy} of a hypersurface $M$ in $\RR^{n+1}$ is the supremum
of 
\[
    \frac1{(4\pi)^{n/2}r^n} \int_{y\in M} e^{-|y-x|^2/4r^2}\, d\HH^ny
 \]
 over all $x\in \RR^{n+1}$ and $r>0$.
Because $\partial K$ is a part of a limit flow at $(p,t)$, its entropy is at most the Gaussian
density of the original flow at the point $p$:
\begin{equation}\label{EntropyBound}
    \operatorname{Entropy}(\partial K) \le \Theta(p) \le d_m.
\end{equation}
(This follows easily from Huisken's monotonicity.)
On the other hand,  $\partial K$ forms an $\Ss^k\times \RR^{n-k}$ singularity
under mean curvature flow.  (If this not clear, apply Huisken's Theorem~\cite{Huisken1} to see
that the mean curvature flow starting with $C$ collapses to a round point, and then cross
that flow with $\RR^{n-k}$ to get a mean curvature flow starting with $K$ and collapsing
to a $(n-k)$-space with an $\Ss^k\times \RR^{n-k}$ singularity.)
By Huisken's monotonicity,
\[
   d_k \le \operatorname{Entropy}(\partial K),
\]
so $d_k \le d_m$  by~\eqref{EntropyBound}.
Thus $k\ge m$ (by~\eqref{DensitiesDecrease}), which implies that the $j^{\rm th}$ homotopy group of $\partial K$ 
 (which is diffeomorphic to $\Ss^k\times \RR^{n-k}$) is trivial for all $j<m$.
\end{proof}

\section{The Main Theorem}
We begin by recalling some topological terminology.
Suppose that $Y$ is a topological space and that $X$ is a subset of $Y$.  We write
\[
    F: (\BB^k, \partial \BB^k) \to (Y, X)
\]
to indicate that $F$ is a continuous map of the pair $(\BB^k, \partial \BB^k)$ into $(Y,X)$, i.e, a
continuous map of $\BB^k$ into $Y$ such that $F(\partial \BB^k) \subset X$.
Two such maps
\[
    F, G: (\BB^k, \partial \BB^k) \to (Y,X)
\]
are called homotopic in $(Y,X)$  provided there is a homotopy
$H:\BB^k\times [0,1]\to Y$ from $F$ to $G$ such that
\[
  H(\cdot, s) : (\BB^k, \partial \BB^k) \to (Y,X)
\]
for all $s\in [0,1]$.

We say that the pair $(Y,X)$ is {\bf $\mathbf{m}$-connected} if for every $k\le m$, every continuous map
\[
     F: (\BB^k, \partial \BB^k) \to (Y,X)
\]
is homotopic in $(Y,X)$ to a map $G$ whose image $G(\BB^k)$ lies in $X$.

We can now state the main theorem:

\begin{theorem}\label{main}
Let $t\in [0,\infty)\mapsto K(t)$ be a mean curvature flow of compact, mean convex subsets
of a Riemannian manifold $N$.
Suppose that $0<a < b$ and that
each singularity during the the time interval $a\le t < b$ has convex type and has
Gaussian density $\le d_m$,
the Gaussian density of a shrinking $m$-sphere in $\RR^{m+1}$. 

Then the pair $(K(b)^c, K(a)^c)$ is $m$-connected.
\end{theorem}

In particular, the conclusion implies that if a map of $\Ss^k$ (for $k<m$) into $K(a)^c$  is 
contractible in $K(b)^c$, then it is also contractible in $K(a)^c$.  Thus Theorem~\ref{main}
implies Theorem~\ref{IntroTheorem}.

Recall (Proposition~\ref{ConvexType}) that if $N=\RR^{n+1}$ and $\partial K(0)$ is smooth, or
if $n< 7$, then the singularities all have convex type.

Note also that if $\partial K(0)$ is smooth, then we can also allow $a=0$ in Theorem~\ref{main}, 
because the topology cannot
not change before the first singular time.

\begin{proof}[Proof of Theorem~\ref{main}]
The theorem follows immediately from Theorem~\ref{main-abstract} and 
Proposition~\ref{Q-density-relation} below.
\end{proof}

\section{An Abstract Form of the Main Theorem}\label{abstract-form-section}

In this section, we state and prove an abstract version (Theorem~\ref{main-abstract})
of Theorem~\ref{main}.
The abstract version is no harder to prove than the special case, but it has the advantage
of also applying to some variants of mean curvature flow.
For example, for $n<7$, it also applies  
to regions $t\mapsto K(t)$ in $\RR^{n+1}$ (or in an $(n+1)$-dimensional Riemannian
manifold) whose boundaries 
evolve with velocity
given by 
\begin{equation}\label{forcing-term}
\vv = H  + F^\perp
\end{equation}
where the drift $F$ is a smooth, time-independent vectorfield on the ambient
space.  The relevant analog
of mean convexity is the condition that at time $0$, the velocity $\vv$ is everywhere nonzero
and points into the region $K(0)$. 

(Existence of a set-theoretic solution 
of~\eqref{forcing-term} may be proved as in~\cite{Ilmanen-levelsetmanifold} or~\cite{White-Topology}.  
Existence of an associated flow of varifolds may be proved by the 
methods of Evans-Spruck~\cite{ES}, 
which are adapted to general ambient manifolds in~\cite{Ilmanen-levelsetindiana}. In particular,
as in \cite{Ilmanen-levelsetindiana}*{\S5}, one gets the set-theoretic flow by taking a limit of of flows
of moving graphs in $N\times\RR$ and then slicing to get the flowing sets in $N$.
  Taking the corresponding varifold limits gives the appropriate varifold flow
in $N\times\RR$, from which one gets the varifold flow in $N$ by slicing.
Once one knows (or assumes) existence of the set-theoretic flow and of the associated flow
of varifolds, the partial regularity proofs in~\cite{White-Size} 
carry over with only minor and straightforward modifications, 
as does the proof in~\cite{White-Nature}  that the singularities
have convex type when $n<7$.  The proof in~\cite{White-Subsequent} of convex-type in 
 $\RR^{n+1}$ 
 for $n\ge 7$ is more delicate and does
not seem to generalize to flows~\eqref{forcing-term} with a drift term.)

\begin{definition}
Suppose $K$ is a closed set in the interior\footnote{For now, the reader may as well assume that
the ambient manifold has no boundary.  In \S5, we will consider sets $K$ that contain
a portion of the boundary of the ambient manifold.}
of a smooth $(n+1)$-dimensional Riemannian manifold.
A point in $K$ is a {\bf regular} point of $K$ if it is an interior point of $K$ or if it is a boundary
point with a neighborhood $U$ such that $K\cap U$ is smoothly diffeomorphic to a closed halfspace
in $\RR^{n+1}$.
The {\bf singular set} $\sing(K)$ of $K$ is the set of points in $K$ that are not regular points
of $K$.
\end{definition}

Note that the singular set of $K$ is a closed subset of $\partial K$.

The quantity $Q(K)$ introduced in the following definition may at first seem peculiar, but
for mean convex mean curvature flow, the condition $Q(K(t))\ge m$ is closely related to 
the condition ``the singularities at $t$ have Gaussian densities $\le d_m$". 
See Proposition~\ref{Q-density-relation} and the discussion following.

\begin{definition}\label{Q-definition}
Suppose $K$ is a closed set in the interior of a smooth $(n+1)$-dimensional Riemannian
manifold.
We define $Q(K)$ to be the largest integer $m$ with the following properties:
\begin{enumerate}[\upshape (a)]
\item\label{Qa} The singular set $\sing(K)$ has Hausdorff dimension $\le n-m$.
\item\label{Qb}
  Let $x_i$ be a sequence of points in the interior of $K$
   converging to a point in $\partial K$.
 Translate $K$ by $-x_i$ and dilate by $1/\dist(x_i, \partial K)$ to get $K_i$.
 Then a subsequence of the $K_i$ converges to a convex subset $K'$ of $\RR^{n+1}$ with
  smooth boundary, and the convergence is smooth on bounded sets.  
\item\label{Qc} If $K'$ is as in~\eqref{Qb}, then $\partial K'$ 
   has trivial $k^{\rm th}$ homotopy for every $k < m$.
\end{enumerate}
If no such integer exists, we let $Q(K)=-\infty$.
\end{definition}

Note that~\eqref{Qb} implies
\begin{equation}
\begin{aligned}\label{closest}
 &\text{If $x\in\operatorname{interior}(K)$ and if $y$ is point in $\partial K$ closest to $x$,}
 \\
&\text{then $y$ is a regular point of $K$.}
\end{aligned}
\end{equation}
(If this is not clear, consider a sequence of points $x_i$ lying on the geodesic between $x$ and $y$ and converging to $y$.  The limit set $K'$ in~\eqref{Qb} must be a halfspace, 
and the smooth convergence in~\eqref{Qb}
then implies that $y$ is a regular point.)

Note also that if $K$ has no interior, then~\eqref{Qb} and~\eqref{Qc} are vacuously true, 
and $\sing(K)=K$, so in that case
$Q(K)$ is the largest integer less than or equal to  $n-\dim(\sing(K))$.

The following proposition describes for mean curvature flow
how $Q(K(t))$ is related to the
Gaussian densities of the singularities at time $t$:

\begin{proposition}\label{Q-density-relation}
Let $t\in [0,T] \mapsto K(t)$ be a mean curvature flow of mean-convex regions in the interior of a smooth Riemannian
$(n+1)$-manifold.   If $t\in (0,T]$ and if the singularities at time $t$ all have convex type
with Gaussian densities $\le d_m$, then $Q(K(t))\ge m$.
\end{proposition}

\begin{proof} The result follows immediately from Proposition~\ref{NearestPoint}, 
Corollary~\ref{SingularSetSizeCorollary}, and Proposition~\ref{MainBlowUpProposition}.
\end{proof}

For mean convex mean curvature flow, 
$Q(K(t))$ will typically equal the smallest $m$ such that there is a singularity at time $t$ with Gaussian
density $d_m$.  However, there are degenerate situations in which $Q(K(t))$ is strictly less than that $m$.
For example, at the singular time for the doubly-degenerate neckpinch in $\RR^3$
 mentioned in the introduction, $K(t)$ is a single point and thus $Q(K(t))=2-0=2$, but
the Gaussian density at that singularity is $d_1$, not $d_2$.

\begin{theorem}\label{main-abstract}
Let $t\in [a,b]\mapsto K(t)$ be a one-parameter family of compact sets in the interior of a smooth, Riemannian 
$(n+1)$-manifold such that
\[
    \text{$K(t)\subset K(T)$ for $a\le T\le t < b$}
\]
and such that the boundaries $M(t):=\partial K(t)$ form a partition of 
$K(a)\setminus \interior(K(b))$.
Assume that 
\begin{equation}\label{Q-hypothesis-main-new}
  \text{$Q(K(T))\ge m$ for each $T\in [a,b)$}.
\end{equation}
Then the pair $(K(b)^c, K(a)^c)$ is $m$-connected.
\end{theorem}

\begin{remark}\label{time-function}
Note that the hypotheses in the first sentence of the theorem imply
that there is a continuous function
$
   \tau: K(a) \to \RR
$
such that
\begin{align*}
&\text{$M(t) =\{x: \tau(x)=t\}$ for $t\in [a,b)$, and} \\
&\text{$K(t)  =\{x: \tau(x)\ge t\}$ for $t\in [a,b]$.}
\end{align*}
If $\tau(x)\in [a,b]$, then $\tau(x)$ is the time at which the moving surface
$\partial K(t)$ passes through the point $x$.
\end{remark}

\begin{proof}
Let $k\le  m$ and let
\[
   F_0: (\BB^k, \partial \BB^k)  \to  (K(b)^c, K(a)^c).
\]
be a continuous map.
Let $\FF$ be the set of all continuous maps 
\[
  F : (\BB^k, \partial \BB^k) \to (K(b)^c, K(a)^c)
\]
such that $F$ is homotopic in $(K(b)^c, K(a)^c)$ to $F_0$.
We must show that $\FF$ contains a map whose image lies in $K(a)^c$, i.e.,
a map whose image is disjoint from $K(a)$.

Equivalently, if $J$ is the set of $t\in [a,b]$ such that $\FF$ contains a map $F$ whose image is
disjoint from $K(t)$, then we must show that $a\in J$.

We will prove that $J=[a,b]$ (and therefore that $a\in J$) by proving the following
four statements:
\begin{enumerate}[\upshape (i)]
\item\label{first} $b\in J$.
\item\label{second} $J$ is a relatively open subinterval of $[a,b]$.
\item\label{third} If $T$ is in the closure of $J$, then $\FF$ contains a map $F$ whose image
 is contained in the union of $K(T)^c$ and the regular part of $\partial K(T)$.
\item\label{fourth} If $T$ is in the closure of $J$, then $T\in J$.
\end{enumerate}
Statements~\eqref{first}, \eqref{second}, and~\eqref{fourth} 
imply that $J$ is a nonempty subinterval of $[a,b]$ that is both open and closed in 
 $[a,b]$,
and therefore that $J$ is all of $[a,b]$, as desired.

Statement~\eqref{first} is trivially true (since $F_0\in \FF$ and $F_0(\BB^k)$ is disjoint from $K(b)$.)

Next we prove statement~\eqref{second}.  For  $F\in \FF$, let
\[
    J_F:=\{t\in [a,b]: F(\BB^k)\cap K(t) = \emptyset\}.
\]
Note that 
\begin{equation}
    J = \cup_{F\in \FF} J_F.
\end{equation}
By definition of $\FF$, the set $J_F$ contains $b$.
Since the $K(t)$'s are nested, if $a\le t\le t' \le b$ and if $t$ is in $J_F$, then $t'$ is also in $J_F$.
Thus $J_F$ is an interval containing $b$.
We claim that $J_F$ is relatively open in $[a,b]$.  If $a\in J_F$, then $J_F=[a,b]$, which is certainly
relatively open in $[a,b]$.  Thus suppose $a\notin J$, i.e., that $K(a)$ intersects $F(\BB^k)$.
Note that there is a last time $t$ such that $K(t)$ intersects $F(\BB^k)$. (Indeed, $t=\max\{\tau(x): x\in F(\BB^k)\}$,
where $\tau(\cdot)$ is the function in Remark~\ref{time-function}.)
Then $J_F=(t,b]$, which is relatively open in $[a,b]$.  We have shown that each $J_F$
is a relatively open subinterval of $[a,b]$ containing $b$.   Hence their union $J$
is also such a subinterval of $[a,b]$.  This proves statement~\eqref{second}.

Next we observe that statement~\eqref{third} implies statement~\eqref{fourth}.  
For suppose $T$ is in the closure of $J$.
Then, assuming that statement~\eqref{third} holds, 
 $\FF$ contains a map $F$ that lies in the union of $K(T)^c$ with the regular
part of $\partial K(T)$.
Now we simply push $F(\BB^k)$ into $K(T)^c$ by pushing it (where it touches the regular part
of $\partial K(T)$) in the direction of the outward unit normal to $K(T)$.
Thus $T\in J$, which completes the proof
that statement~\eqref{third} implies statement~\eqref{fourth}.

(The sentence ``now we simply push\dots"  may be made more precise as follows.
Let $S= F(\BB^k)\cap \partial K(T)$.   Let $\vv$ be a smooth, compactly supported vectorfield 
defined on the regular part of $\partial K(T)$ such that $\vv$ is nonzero at every point of $S$
and such that at each point, $\vv$ is a nonnegative multiple of the outward unit normal to $\partial K(T)$.
  Now extend $\vv$ to be a smooth vectorfield on the ambient space that vanishes outside of $K(a)$.
The flow generated by $\vv$ homotopes $F$ to a map in $\FF$ whose
image is disjoint from $K(T)$.)

It remains only to show statement~\eqref{third}.  Suppose $T\in [a,b)$ is in the closure of $J$. 
Let $\eps>0$ (to be specified later).  By statements~\eqref{first} and~\eqref{second},
 there exist $T^*\in J\cap(T,b]$ arbitrarily close
to $T$.   Choose such a $T^*$ sufficiently close to $T$ that every point in $K(T)\setminus K(T^*)$
is within distance $<\eps$ of $\partial K(T)$.  
(This is possible by the continuity of the function $\tau(\cdot)$ in Remark~\ref{time-function}.)

Since $T^*\in J$, there is a map $F\in \FF$ such that $F(\BB^k)$ is disjoint from $K(T^*)$.   
We may assume that $F$ is smooth since the $C^\infty$ maps are dense in the set of continuous maps.
Now
\[
    \dim(\sing(K(T))) \le n- Q(K(T)) \le n-m,
\]
and therefore since $k\le m$,
\[
   \dim(\sing(\partial K(T))) + k \le n < n+1.
\]
Consequently, we may assume, by putting $F$ in general position, that $F(\BB^k)$
contains no singular points of $\partial K(T)$.   (See the appendix if this is not clear.)

We will construct a map $G$ from $\BB^k$ such that the image  of $G$ is contained in $K(T)^c$
together with the regular part of $\partial K(T)$.   We will also construct a homotopy
 from $F$ to $G$ in $(K(b)^c, K(a)^c)$.  The homotopy shows that $G\in \FF$, 
thus establishing statement~\eqref{third}.

Let $\Omega\subset \BB^k$ be the inverse image under $F$ of the interior of $K(T)$.
We now describe the construction of the map $G$ on the open set $\Omega$.

First some terminology.  
Recall that  a $d$-simplex is the convex hull of $(d+1)$ points in a Euclidean space provided
those $(d+1)$ points do not lie in any affine subspace of dimension $<d$.
The points are called vertices of the simplex.
If the distance between each pair of vertices is $1$, we say that the
simplex is {\bf standard}.  
Note that any two $d$-simplices are affinely isomorphic.
In particular, given any $d$-simplex $\Delta$, there is an affine bijection $\sigma: \Delta \to \Delta_s$
from $\Delta$ to a standard simplex $\Delta_s$. 
 We define the {\bf standardized distance} $d_s(\cdot, \cdot)$ 
on $\Delta$ by
\[
      d_s(x,y) = | \sigma(x) - \sigma(y)|.
\]
Given a map $F$ from $\Delta$ into a metric space $Z$, we define the {\bf standardized Lipschitz
constant} $\Lip_s(F)$ of $F$ to be the Lipschitz constant of $F$ with respect to the standardized
distance on $\Delta$:
\[
    \Lip_s(F)
    =
    \sup_{x\ne y} \, \frac{ \dist( F(x), F(y) )}{  d_s(x,y) }.
\]

We now describe the map $G$ on the portion $\Omega$ of $\BB^k$.  (Later we will extend
$G$ to all of $\BB^k$ by letting $G=F$ on $\BB^k\setminus \Omega$.)
First, triangulate $\Omega$.
By refining the triangulation, we may assume that
for each simplex $\Delta$ of the triangulation,
\[
    \diam(F(\Delta)) < \eps \, \dist(F(\Delta), \partial K(T)).
 \]
 Here $\dist(X,Y)$  denotes the infimum of $\dist(x,y)$ among all $x\in X$ and $y\in Y$.

We define $G$ on $\Omega$ inductively by defining it first
on the $0$-skeleton of the triangulation of $\Omega$, then
on the $1$-skeleton, and so on.
For each vertex $v$ in the $0$-skeleton, we choose a point $q\in \partial K(T)$ that minimizes
$\dist(q, F(v))$, and we then let $G(v)$ be that chosen $q$.  
Note that $q$ is a regular point of $\partial K(T)$ by~\eqref{closest}.
Having defined $G$ on the $(j-1)$-skeleton of $\Omega$, we extend it to the $j$-skeleton as follows.  For each
$j$-simplex $\Delta$ in the triangulation, we choose a map 
\[
   g: \Delta \to \partial K(T)
\]
that minimizes $\Lip_s(g)$ among all maps $g:\Delta \to \partial K(T)$ such
that $g=G$ on $\partial \Delta$.  Having chosen such a $g$, we let $G(x)=g(x)$
for $x\in \Delta$.  (In Lemma~\ref{lemma} below, any map $G$ constructed by this inductive procedure will be called 
   ``$F$-optimal".)

Of course we must check that the procedure does not break down in going from the $(j-1)$-skeleton
to the $j$-skeleton.
 That it does not break down is proved below in Lemma~\ref{lemma} (provided $\eps>0$ is sufficiently small).
The lemma shows (for all sufficiently small $\eps>0$) that:
\begin{enumerate}[\upshape (1)]
\setcounter{enumi}{\value{equation}}
\item  $G(\Omega)$ lies in the regular part of $\partial K(T)$, and
\item\label{C^0-close} $\dist(F(x),G(x)) \le C \dist (F(x), \partial K(T))$ for all $x\in \Omega$.  (See~\eqref{diam-of-union}
  in the lemma.)
\setcounter{equation}{\value{enumi}}
\end{enumerate}

By~\eqref{C^0-close}, the map $G$ extends continuously to $\BB^k$ 
 by setting $G(x)=F(x)$ for $x\in \BB^k\setminus \Omega$.

Now define a homotopy $H:\BB^k\times[0,1]\to K$ from $F$ to $G$ by setting
\[
   H(x,s) = (1-s)F(x) + sG(x)
\]
if the ambient space is Euclidean.   More generally, we define $H$ by letting $H(x, \cdot): [0,1] \to K(T)$ be the unique shortest 
geodesic (parametrized with constant speed) joining $F(x)$ to $G(x)$.
(By~\eqref{C^0-close}, the shortest geodesic will be unique if $\eps>0$ is sufficiently small, since 
         $\dist(F(x),\partial K(T))<\eps$.)

It remains only to show that (if $\eps>0$ is sufficiently small) the image of $H$ is disjoint from $K(b)$,
i.e., that for $x\in \Omega$, the geodesic from $F(x)$ to $G(x)$ is disjoint from $K(b)$.
Choose $\eps$ with
\[
 0<\eps < \frac{\dist(\partial K(T), K(b))}{C}.
\]
(This is possible since $\partial K(T)$ and $K(b)$ are disjoint.)
Thus by~\eqref{C^0-close},
\[
  \dist(F(x), G(x)) < \dist(\partial K(T), K(b)).
\]
This means that the geodesic from $G(x)$ (which is in $\partial K(T)$) to $F(x)$ is too short to reach $K(b)$.
Thus that geodesic is disjoint from $K(b)$.

We have proved that the image of the homotopy $H$ is disjoint from $K(b)$.
The homotopy proves that $G\in \FF$.  This completes the proof of  Theorem~\ref{main-abstract}.
\end{proof}

We now turn to the lemma that was used in the proof of Theorem~\ref{main-abstract}.
First we need some terminology.
Fix a $T>0$ and let $K=K(T)$.
Let $X$ be a simplicial complex and let $F$ be a map from $X$ to $K$.
We say that a map
$G: X\to \partial K$ is {\bf $F$-optimal}
 provided:
\begin{enumerate}
\item For each vertex $v$ of $X$, $G(v)$ realizes the minimum distance from a point in $\partial K$ to $F(v)$:
\[    
     \dist(F(v), \partial K) = \dist(F(v), G(v)).
\]
\item For each simplex $\Delta$ of $X$, the restriction $G|\Delta$ is a $\Lip_s$-minimizing map
from $\Delta$ to $\partial K$.  That is, if $g:\Delta\to \partial K$ is any map
such that $g|\partial \Delta=G|\partial \Delta$, then
\[
  \Lip_s(G|\Delta) \le \Lip_s(g).
\]
\end{enumerate}

\begin{lemma}\label{lemma}
Let $K$ be a compact subset of the interior of a smooth, $(n+1)$-dimensional manifold.
Let $\Delta$ be a simplex of dimension $k \le  Q(K)$.
Then there is an $\eps>0$ and a $C<\infty$ with the following property.
If $F: \Delta \to K$ is a map such that
\[
   \diam(F(\Delta))  <  \eps \dist(F(\Delta), \partial K) 
\]
and such that
\[
   \dist(F(\Delta), \partial K) < \eps,
\]
then each $F$-optimal map from $\partial \Delta$ to $\partial K$ extends
to an $F$-optimal map $G$ from $\Delta$ to $\partial K$, and for any such extension
$G$, 
\[
\Lip_s(G)   \le C  \, \dist( F(\Delta), \partial K),
\]
and
\begin{equation}\label{diam-of-union}
    \diam( F(\Delta)\cup G(\Delta))   \le C  \, \dist( F(\Delta), \partial K).
\end{equation}
\end{lemma}

We may assume that the simplex $\Delta$ is standard since the statement of the theorem is not affected
by affine reparametrizations of the domain.   For purposes of proof, it is convenient to restate the lemma
as follows:

\begin{lemma}\label{lemma-restated}
Let $K$ be as in Lemma~\ref{lemma}, and let $\Delta$ be 
a standard simplex of dimension $k \le  Q(K)$.
Let $\eps_i\to 0$, and suppose that $F_i:\Delta\to K$ is a sequence of maps such that
\begin{equation}\label{diameter-tiny}
    \diam(F_i(\Delta)) \le \eps_i \dist( F_i(\Delta), \partial K)
\end{equation}
and such that
\begin{equation}\label{distance-tiny}
  \dist( F_i(\Delta), \partial K) < \eps_i.
\end{equation}
Suppose also that $\Gamma_i: \partial \Delta\to \partial K$ is a sequence of $F_i$-optimal
maps.  Then for all sufficiently large $i$, there exists an $F_i$-optimal map $G_i: \Delta\to \partial K$
that extends $\Gamma_i$, 
and such a $G_i$ must (for all sufficiently large $i$) have the following properties:
\begin{enumerate}[\upshape (i)]
\item\label{G-in-regular-part} $G_i(\Delta)$ is contained in the regular part of $\partial K$.
\item\label{ratios-bounded} The quantities
\[
   \frac{\Lip G_i}{\dist(F_i(\Delta), \partial K)}
\]
(if $k>0$) and
\[ 
    \frac{\diam(F_i(\Delta) \cup G_i(\Delta))}{ \dist( F_i(\Delta), \partial K) }
\]
are bounded above as $i\to\infty$.
\end{enumerate}
\end{lemma}

\begin{proof}
We prove it by induction on the dimension of $\Delta$.

If $\Delta$ is $0$-dimensional, it is a single point $p$.
Let $G_i(p)$ be a point in the interior of $\partial K$ such that 
\[
    \dist(F_i(p), G_i(p)) = \dist(F_i(p), \partial K).
\]
Since $Q(K)>-\infty$, this implies that $G_i(p)$ is a regular point (see~\eqref{closest}), 
so~\eqref{G-in-regular-part}
 holds. The two ratios in~\eqref{ratios-bounded} are trivially equal to $0$ and $1$, 
 so~\eqref{ratios-bounded} also holds. This completes the proof of the lemma when
 $\Delta$ is $0$-dimensional.

Now suppose that $1\le k=\dim(\Delta)\le Q(K)$.
By induction, we may assume that the lemma is true for each face of $\Delta$.
Let $p_i$ be a point in $F_i(\Delta)$ that minimizes the distance from $F_i(p_i)$ to $\partial K$.
  
Translate $K$  by $-F(p_i)$ and dilate by
\[
     \lambda_i = \frac1{\dist(F_i(p), \partial K)}
\]
 to get a set $K_i'$.
Let $F_i':S\to K_i'$ and $\Gamma_i': \partial \Delta\to \partial K_i'$ be the maps
corresponding to $F_i$ and $\Gamma_i$.  Note that
\begin{equation}\label{0-in-image}
    0\in F_i'(\Delta) \subset K_i'
\end{equation}
and that
\begin{equation}\label{NormalizedDistToBoundary}
   1 = \dist( 0, \partial K') = \dist(F_i'(\Delta), \partial K_i').
\end{equation}

By passing to a subsequence, we may assume that the $K_i'$ converge smoothly
to a convex set $K'$ with
\begin{equation}\label{OriginInK'}
    \text{$0\in K'$ and $\dist(0, \partial K') = 1$}. 
\end{equation}

By~\eqref{diameter-tiny} and~\eqref{NormalizedDistToBoundary},
\[
  \diam (F_i'(\Delta)) \le \eps_i \, \dist( F_i'(\Delta), \partial K'_i)  = \eps_i \to 0,
\]
so by~\eqref{0-in-image},
\begin{equation}\label{converge-to-zero}
   \text{$F_i'(\cdot)\to 0$ uniformly.}
\end{equation}
Thus by~\eqref{OriginInK'},
\begin{equation}\label{separation-one}
   \dist(F_i'(\cdot), \partial K_i') \to 1 \quad \text{uniformly.}
\end{equation}

By induction we can assume that~\eqref{ratios-bounded}
 holds for the restrictions of $F_i$  and $\Gamma_i$ to each face $\Delta^*$ of $\Delta$.
Thus
\[   
      \Lip (\Gamma_i \vert \Delta^*) \le c \dist( F_i(\Delta^*), \partial K)
\]
for some constant $c$, which implies by rescaling that
\[
      \Lip (\Gamma'_i \vert \Delta^*) \le c \dist( F'_i(\Delta^*), \partial K_i').
\]
By~\eqref{converge-to-zero} and~\eqref{separation-one}, the right hand side
tends to $c$, so
\begin{equation}\label{GammasAreLipschitz}
   \limsup_i \left( \Lip(\Gamma'_i \vert \Delta^*) \right) \le c.
\end{equation}

If $v$ is a vertex of $\Delta$, then $\Gamma_i'(v)$ is a point in $\partial K_i'$ closest to $F_i'(v)$.
Since  since $F_i'(\cdot)\to 0$ and since $K_i'\to K'$ smoothly, this implies that
\begin{equation}\label{G'bounded}
  \limsup_{i\to\infty} \, \dist(\Gamma_i'(v), 0) = \dist(\partial K', 0)  = 1.
\end{equation}

By~\eqref{GammasAreLipschitz} and~\eqref{G'bounded}, 
the $\Gamma_i'$ form an equicontinuous family, so after passing to a subsequence, we can
assume that the $\Gamma_i'$ converge uniformly to a Lipschitz map
\[
  \Gamma': \partial \Delta \to \partial K'.
\]
Now $\partial K'$ is smooth.
Also, $k=\dim(\Delta) \le Q(K)$, so by definition of $Q(K)$,  
the $(k-1)$-dimensional homotopy of $\partial K'$ is trivial.
Thus the map $\Gamma'$
extends to a Lipschitz map $G': \Delta \to \partial K'$.  

By  the smooth convergence $K_i'\to K'$ and by the bounded Lipschitz norm convergence
$\Gamma_i' \to \Gamma'$, it follows that (for all sufficiently large $i$) there is a Lipschitz map
\[
      G_i' : \Delta \to \partial K_i'
\]
such that $G_i'$ extends $\Gamma_i'$ and such that
\begin{equation}\label{delta-i}
    \Lip(G_i') \le \Lip(G') + \delta_i
\end{equation}
where $\delta_i\to 0$.   We may assume that $G_i': \partial \Delta \to K_i'$ is the extension
of smallest Lipschitz norm. (This minimizing extension exists because $\partial K_i'$ is compact.)
By passing to a subsequence, the $G_i'$ converge uniformly to a limit map, which
we may assume to be $G'$.  (Otherwise redefine $G'$ to be that limit map.)

In particular, the smooth convergence $\partial K_i' \to \partial K'$ implies that $G_i'$
maps $\Delta$ to the regular part of $\partial K'_i$ (if $i$ is sufficiently large).

Note that 
\[
   \frac{\Lip(G_i)}{\dist(F_i(\Delta),\partial K)}
   =
   \frac{\Lip(G'_i)}{\dist(F_i'(\Delta),\partial K_i')}
   =
   \frac{\Lip(G'_i)}{1} 
\]
which is bounded as $i\to\infty$ by~\eqref{delta-i}.

Similarly we have
\begin{equation}\label{DiameterEstimate}
   \frac{\diam(F_i(\Delta)\cup G_i(\Delta))}{\dist(F_i(\Delta), \partial K)}
   =
   \frac{\diam(F'_i(\Delta)\cup G'_i(\Delta))}{\dist(F'_i(\Delta), \partial K_i')}
   =
   \frac{\diam(F'_i(\Delta)\cup G'_i(\Delta)}1
\end{equation}
which converges to $\diam(\{0\}\cup G'(\Delta))$ as $i\to\infty$ (since $F'_i\to 0$ and $G'_i\to G'$ uniformly.)
In particular, \eqref{DiameterEstimate} is bounded as $i\to\infty$.
\end{proof}

\section{Manifolds with Boundary}\label{Boundaries}

So far in this paper, the moving hypersurfaces $\partial K(t)$ under consideration
have been hypersurfaces without boundary.   Now we consider the case of hypersurfaces
with boundary, the motion of the boundary being prescribed and the motion away
from the boundary being by mean curvature flow (or possibly by other analogous flows.)

\begin{definition}
Let $N$ be a smooth $(n+1)$-dimensional manifold-with-boundary.
Let $K$ be a closed subset of $N$.
A point $p\in K$ is called a {\bf regular point} of $K$ provided
\begin{enumerate}
\item $p$ is an interior point of $K$, or
\item $p\in N\setminus \partial N$ and $N$ has a neighborhood $U$
of $p$ such that $K\cap U$ is diffeomorphic to a closed half-space in $\RR^{n+1}$, or
\item $p\in \partial N$ and $N$ has a neighborhood $U$ of $p$ for which there
is a diffeomorphism that maps $U$ onto $\{x\in \RR^{n+1}: x_1\ge 0\}$ and 
that maps $K\cap U$ onto $\{x\in \RR^{n+1}: x_1\ge 0, x_2\ge 0\}$.
\end{enumerate}
\end{definition}
Points in $K$ that are not regular points are called singular points of $K$.

The following theorem should be thought of as a theorem about a moving hypersurface-with-boundary.
At time $t$, the hypersurface is 
\[
       M(t)=\partial K(t) = K(t)\cap \overline{N\setminus K(t)},
\]
and its
boundary is $\Gamma(t):=M(t)\cap \partial N$.  In practice, the initial surface would be prescribed by prescribing
$K(0)$, and the motion of the boundary would be prescribed by prescribing  $\Gamma(t)$ or, equivalently, by prescribing $K(t)\cap N$. The geometric flow
would then determine the moving region $K(t)$ or, equivalently, the moving hypersurface $M(t)$.

We first give an abstract version of the main theorem of this section
  (afterwards, in Theorem~\ref{main-boundary}, we specialize
to mean curvature flow):

\begin{theorem}\label{main-abstract-boundary}
Let $t\in [a,b]\mapsto K(t)$ be a one-parameter family of compact subsets of a smooth, 
$(n+1)$-dimensional Riemannian manifold-with-boundary $N$, and let 
\[
  M(t)=\partial K(t)=K(t)\cap \overline{N\setminus K(t)}.
\]
Assume that there is a collared neighborhood $U\subset N$ of $\partial N$
such that each $M(t)\cap U$ is a smooth, 
embedded manifold-with-boundary, the boundary being $M(t)\cap \partial N$, and 
that $M(t)\cap U$ depends smoothly on $t$ for $t\in [a,b]$.
Assume that $M(t)\cap U$ is never tangent to $\partial N$.
Assume also that
\begin{enumerate}[\upshape (1)]
\item\label{nested} $K(t)\subset K(T)$ for $T\le t$.
\item\label{union} $K(a)\cap K(b)^c = \cup_{a\le t < b} M(t)$.
\item\label{disjoint} $M(t)\cap M(T)\subset \partial N$ for $t\ne T$.
\end{enumerate}
If 
\[
 \text{$Q(K(t)) \ge m$ for all $t\in [a,b)$,}
\]
then the pair $(K(b)^c, K(a)^c)$ is $m$-connected.
\end{theorem}

Note that hypothesis~\eqref{disjoint} allows the boundary of $M(t)$ to be fixed or to move.
Note also that the hypotheses imply that the singularities of $K(t)$ lie in the interior of $N$.

If the $M(t)$'s are disjoint, Theorem~\ref{main-abstract-boundary} can be proved
exactly as 
Theorem~\ref{main-abstract} was proved.
In the general case, we can reduce to the case of disjoint $M(t)$'s by replacing $N$ by 
\[
   N':=N\cap \{x : \dist(x,\partial N)\ge \delta\}
\]
for some sufficiently small $\delta >0$, 
and by replacing each $K(t)$ by $K(t)\cap N'$.

(For the proof, it is useful to note that if $F:(\BB^k, \partial \BB^k)\to (K(t)^c, K(a)^c)$, 
then $F$ is homotopic in $(K(t)^c, K(a)^c)$ to a map $G$ whose image lies in the interior
of $N$.  To see this, let $\vv$ be any vectorfield on $N$ that is equal on $\partial N$ to the
unit normal pointing into $N$.  Now flow by that vectorfield for a short time to push $F$ into
the interior of $N$.)

In the case of mean curvature flow, we have the following theorem.
In the statement of the the theorem, the boundary 
   $\Gamma(t)$ of the moving surface $M(t):=\partial K(t)$
is given by giving a region $V(t)$ in $\partial N$ such that $\Gamma(t)=\partial V(t)$.

\begin{theorem}\label{main-boundary}
Let $N$ be a smooth, compact, connected $(n+1)$-dimensional Riemannian manifold with boundary
with $n<7$.
Let $t\in [0,\infty)\mapsto V(t)$ be a smooth, one-parameter family of compact, 
smooth, $n$-dimensional
manifolds with boundary in $\partial N$ such that $V(t')\subset V(t)$ for $t\ge t'$.
Let $K$ be a closed subset of $N$ such that $\partial K$ is a smooth, compact, connected
manifold-with-boundary such that $K\cap \partial N=V(0)$, and such that $\partial K$ is smooth
with mean curvature at each point a nonnegative multiple of the unit normal that points into $K$,
and such that $\partial K$ is nowhere tangent to $\partial N$.

If $\partial K$ is a minimal surface (i.e., has mean curvature $0$ at all points),
assume\footnote{This assumption guarantees that the surface starts moving immediately.} 
also that $V(t)\ne V(0)$ for $t\ne 0$.

Let $t\in [0,\infty)\mapsto M(t)$ be the solution obtained
by elliptic regularization of mean curvature flow such that $M(0)=\partial K$
and such that $\partial M(t)=\partial V(t)$ for all $t$.

Then each $M(t)$ is the boundary in $N$ of a region $K(t)\subset K$.
The singularities of the flow form a compact subset of the interior of $N$
and all have convex type.   

In particular, if the Gaussian densities of the singularities in the time interval $a\le t < b$
are all $\le d_m$,
then $t\mapsto K(t)$ satisfies all the hypotheses of Theorem~\ref{main-abstract-boundary},
and therefore the pair $(K(b)^c, K(a)^c)$ is $m$-connected.

If $n\ge 7$, the Theorem remains true provided the metric on $N$ is flat and provided
$M(t)$ is smooth for some $t\ge b$.
\end{theorem}

The theorem should be true for all $n$ without the somewhat peculiar assumptions
in the last sentence of the theorem.  Those assumptions are needed only because
without them we do not know how to prove that the singularities of the flow
have convex type.

\begin{proof}
Except for the assertion that the pair $(K(b)^c, K(a)^c)$ is $m$-connected, 
this is proved in~\cite{White-Subsequent}.  The $m$-connectivity of $(K(b)^c, K(a)^c)$
then follows by Theorem~\ref{main-abstract-boundary} and Proposition~\ref{Q-density-relation}.
\end{proof}

\section{The topology of the moving regions}\label{duality-section}

The theorems described so far are about the topology of the exteriors of the moving regions $K(t)$.
Using standard duality theorems of topology, we can draw conclusions about the changing topology
of the regions themselves. 

\begin{proposition}\label{top-prop}
Let $X$ be a compact orientable $(n+1)$-dimensional manifold-with-boundary.
Suppose 
\[
   \partial X = A \cup B
\]
where $A$ and $B$ are compact $n$-manifolds-with-boundary, and
that $A\cap B = \partial A= \partial B$.
If $(X,A)$ is $m$-connected, then 
\begin{equation}\label{relative-homology-in-top-prop}
    \operatorname{H}_k(X,B)=0
\end{equation}
for $k>n-m$. 
\end{proposition}

\begin{proof}
Let $p\le m$.
Since $(X,A)$ is $m$-connected, 
\[
 \operatorname{H}_p(X,A) = \operatorname{H}_{p-1}(X,A) = 0.
\]
It follows that $\operatorname{H}^p(X,A)=0$.  (This is easy to prove directly, but it is also a special case of the Universal
Coefficients Theorem \cite{Hatcher}*{Theorem~3.2, p. 195}.)
By the Poincare-Lefschetz Duality Theorem~\cite{Hatcher}*{Theorem~3.43, p. 254},
\[
 \operatorname{H}^p(X,A) \cong \operatorname{H}_{n+1-p}(X,B).
\] 
Thus $\operatorname{H}_{n+1-p}(X,B)=0$.  This holds for every $p\le m$, so~\eqref{relative-homology-in-top-prop} 
holds for every $k>n-m$.
\end{proof}

\begin{theorem}\label{add-conclusion}
Suppose, in Theorems~\ref{main}, \ref{main-abstract}, \ref{main-abstract-boundary}, or~\ref{main-boundary}, 
that the surfaces $\partial K(a)$ and $\partial K(b)$ are smooth.
Then:
\begin{enumerate}[\upshape (1)]
\item\label{relative-homology} $\operatorname{H}_k(K(a), K(b))=0$ for all $k>n-m$.
\item\label{homology-inclusion} The map $\iota_\#: \operatorname{H}_k(K(b))\to \operatorname{H}_k(K(a))$ is an isomorphism for $k>n-m$ and is an injection for $k=n-m$.
\end{enumerate}
\end{theorem}

\begin{proof}
Consider first the case of Theorems~\ref{main} and~\ref{main-abstract}.
Let $A=\partial K(a)$, $B=\partial K(b)$, and $X= K(a)\setminus \interior(K(b))$.
By those theorems, $(K(b)^c, K(a)^c)$ is $m$-connected.
Since $\partial K(a)$ and $\partial K(b)$ are smooth, $m$-connectivity of $(X,A)$ follows easily.  
Thus by Proposition~\ref{top-prop},
\begin{equation}\label{homology-zero}
    \operatorname{H}_k(X, B) = 0
\end{equation}
for $k>n-m$.
But by excision,  $\operatorname{H}_k(X,B)= \operatorname{H}_k(K(a),K(b))$.  This proves assertion~\eqref{relative-homology}.
Assertion~\eqref{homology-inclusion} follows immediately by the long exact sequence for $\operatorname{H}_*(K(a), K(b))$.

Now consider the case of Theorems~\ref{main-abstract-boundary} and~\ref{main-abstract}.
For simplicity, let us suppose that the boundary of the moving surface is fixed:
\[
     \partial M(t) \equiv \Gamma \qquad (t\in [a,b]).
\]
Then one applies Proposition~\ref{top-prop} exactly as in the previous paragraph.  In that paragraph, 
the $A$ and $B$ were disjoint, whereas now $A\cap B = \Gamma$.  (Of course $X$
has corners, so the boundary is not smooth, but $X$ is a topological manifold with boundary,
so Proposition~\ref{top-prop} applies.)

If the boundary of the surface $M(t)$ is not fixed, one can still apply Proposition~\ref{top-prop}:
one lets $A=\partial K(a)$ and
\[
   B = (\partial K(b)) \cup (\cup_{t\in [a,b]} \Gamma(t))
\]
and argues as before to get~\eqref{homology-zero}.
It follows that
\[
    \operatorname{H}_k(X, \partial K(b))=0
\]
since  $(X,B)$ is homotopy equivalent to $(X, \partial K(b))$.
Finally, one uses excision and the long exact sequence for $\operatorname{H}_*(K(a), K(b))$
exactly as before to get~\eqref{relative-homology} and~\eqref{homology-inclusion}.
\end{proof}

\begin{theorem}\label{homology-theorem}
Let $t\mapsto K(t)$ be as in Theorems~\ref{main} or~\ref{main-boundary}, 
but without the assumption about Gaussian density of singularities.

If there is an integral $p$-cycle in $K(a)^c$ that bounds a $(p+1)$-chain in
in $K(b)^c$ but not in $K(a)^c$, then there is a singularity (in the time interval $a< t < b$)
whose Gaussian density is $\ge d_p$.

If $\partial K(a)$ and $\partial K(b)$ are smooth and if there is an integral $q$-cycle
in $K(b)$ that bounds in $K(a)$ but not in $K(b)$, then there is a singularity in the time
interval $a<t<b$ with Gaussian density $\ge d_{n-q-1}$.
\end{theorem}

To illustrate Theorem~\ref{homology-theorem}, 
suppose $K(a)$ is connected but that $K(b)$ is not connected.  Let $x$ and $y$ be points in $K(b)$
that lie in different connected components of $K(b)$.  Then the $0$-cycle $[x]-[y]$ (i.e, the
cycle consisting of the point $x$ with multiplicity $1$ and the point $y$ with multiplicity $-1$)
bounds a $1$-chain in $K(a)$ but not in $K(b)$.  Thus according to Theorem~\ref{homology-theorem}
 there must be a singularity with
Gaussian density $\ge d_{n-1}$.

\begin{proof}
If $R$ is an integral $p$-cycle in $K(a)^c$ that bounds a $(p+1)$-chain $S$ in $K(b)^c$ but does not
bound any chain in $K(a)^c$, then $S$ represents a nonzero element of $\operatorname{H}_{p+1}(K(b)^c, K(a)^c)$,
so $(K(b)^c, K(a)^c)$ is not $(p+1)$-connected, so (by Theorem~\ref{main} or~\ref{main-boundary}),
there must be a singularity  (in the time interval $a<t<b$) whose Gaussian density is $> d_{p+1}$
and therefore $\ge d_p$.

Similarly, if $R$ is an integral $q$-cycle in $K(b)$ that bounds a $(q+1)$-chain $S$ in $K(a)$ but
not in $K(b)$, then $S$ represents an nonzero element of $\operatorname{H}_{q+1}(K(a), K(b))$, so 
by Theorem~\ref{add-conclusion}, there is a singularity with Gaussian density $\ge d_{n-q-1}$.
\end{proof}

\section{The Persistence of Neck-Pinches}\label{persist-section}

\begin{theorem}\label{persist-theorem}
Suppose $1\le m \le n$.  
Let $\FF$ be the family of all compact, mean convex regions in $\RR^{n+1}$
with smooth boundary.  
Let $\FF_m$ be the set of $K\in \FF$ such that 
 the mean curvature flow with initial surface $\partial K$ has
 a shrinking $\Ss^m\times \RR^{n-m}$ singularity.
Then $\FF_m$ has nonempty interior.
\end{theorem}

Of course the meaning of ``interior" depends on the choice of topology on $\FF$.
Here we use the topology in which in $K_i$ converges to $K$ if and only
if $\partial K_i$ converges to $\partial K$ in $C^1$.  (We could use $C^k$
for any $k$ with $1\le k\le \infty$, but $C^1$ gives the strongest result.)

We remark that the degenerate neckpinches mentioned in the introduction
show that $\FF_m$ is not, in general, an open subset of $\FF$.

\begin{proof}
Let $\Cc$ be the set of $K\in \FF$ such that $K^c$ has nontrivial $m^{\rm th}$ homotopy
and such that the mean curvature flow with initial surface $\partial K$ has no singularities
of Gaussian density $\ge d_{m-1}$ (or, equivalently, such that each singularity of the flow
has Gaussian density $<d_{m-1}$).   Since the Gaussian density at a spacetime point $(x,t)$
is upper semicontinuous as a function of the spacetime point and of the flow, the set $\Cc$
is an open subset of $\FF$.
By Corollary~\ref{corollary}, the resulting mean curvature flow
has an $\Ss^k\times \RR^{n-k}$ singularity for some $k\le m$.
But by definition of $\Cc$, we have $d_k<d_{m-1}$, which implies (see~\eqref{DensitiesDecrease})
that $k>m-1$ and thus
that $k=m$.  Hence $\Cc$ is a subset of $\FF_m$.

We have shown that $\Cc$ is open and that $\Cc$ is contained in $\FF_m$.  It remains
only to show that $\Cc$ is nonempty.   Let $S$ be a round $(n-m)$-sphere in $\RR^{n+1}$.
For $\eps>0$, let $K(\eps)$ be the set of points at distance $\le \eps$ from $S$.
If $\eps>0$ is sufficiently small, $K(\eps)$ will be mean convex. Fix such an $\eps$.
By symmetry, $K(\eps)$ collapses under mean curvature flow to a round $(n-m)$-sphere, from
which it easily follows that the singularities are all $\Ss^m\times\RR^{n-m}$ singularities.
Also, $K(\eps)^c$ has nontrivial $k^{\rm th}$ homotopy, so $K(\eps)\in \Cc$ and therefore
$\Cc$ is nonempty.
\end{proof}

\section{Appendix}

Here we give a proof of the general position principle used in the proof of Theorem~\ref{main-abstract}.

\begin{proposition}
Let $N$ be a smooth $d$-dimensional manifold without boundary
and let $S$ be a subset of $N$ with Hausdorff $(d-k)$-dimensional measure $0$.
Then the collection $\mathcal{C}$ of smooth maps $F:\BB^k\to N$ such that $F(\BB^k)\cap S=\emptyset$
is dense in the set of all smooth maps from $\BB^k$ to $N$.
\end{proposition}

\begin{proof}
First consider the case $N=\RR^d$.  
Let $F:\BB^k\to N$ be a smooth map.  We will prove the proposition in this case by showing
\begin{equation}\label{almost-every-v}
\text{If $F\in C^\infty(\BB^k,\RR^d)$, then $F(\cdot)+v\in \mathcal{C}$ for almost every $v\in \RR^d$.}
\end{equation}
The set $\Pi^{-1}(S)=\BB^k\times S$ has $k+(d-k)$-dimensional (i.e., $d$-dimensional) measure~$0$.
(Here $\Pi: \BB^k\times \RR^d\to \RR^d$ is the projection map.)
Therefore its diffeomorphic image $\phi(\Pi^{-1}(S))$ under the diffeomorphism
\[
   \phi: (x,y)\in \BB^k\times \RR^n \mapsto (x,y-F(x))
\]
has $d$-dimensional measure $0$.  Hence the projected image $\Pi(\phi(\Pi^{-1}(S)))$ of $\phi(\Pi^{-1}(S))$ 
in $\RR^d$
has Lebesgue measure $0$:
\begin{equation}\label{LebesgueNullset}
       \LL^d(\Pi(\phi(\Pi^{-1}(S))))=0.
\end{equation}
Now
\begin{equation}\label{first-iff}
v\in \Pi(\phi(\Pi^{-1}(S)))
\iff
\text{$(x,v)\in \phi(\Pi^{-1}(S))$ for some $x\in \BB^k$},
\end{equation}
and
\begin{equation}\label{second-iff}
\begin{aligned}
(x,v)\in \phi(\Pi^{-1}(S))
&\iff
\phi^{-1}(x,v)\in \Pi^{-1}(S) 
\\
&\iff
(x,F(x)+v)\in \Pi^{-1}S
\\
&\iff
F(x)+v \in S
\\
\end{aligned}
\end{equation}
The desired conclusion~\eqref{almost-every-v} follows immediately
from~\eqref{LebesgueNullset}, \eqref{first-iff}, and ~\eqref{second-iff}.
This completes the proof in the case $N=\RR^d$.

For a general manifold $N$, we may assume that $N$ is a smooth
submanifold of some Euclidean space $\RR^{d+j}$.
Let $U$ be an open subset of $\RR^{d+j}$ that contains $N$ and for which
the nearest point retraction $\pi:U\to N$ exists and is smooth.

Let $F: \BB^k\to N$ be a smooth map, and let $\delta=d(F(\BB^k), U^c)$.
Now the set $\pi^{-1}(S)$ has $(d-k)+j$ dimensional measure $0$,
or, equivalently $(d+j)-k$-dimensional measure $0$.  Thus by~\eqref{almost-every-v},
the map $F(\cdot)+v$ has image disjoint from $\pi^{-1}(S)$ for almost every $v\in \RR^{d+j}$ with $|v|<\delta$.
Therefore the map $\pi(F(\cdot)+v)$ has image disjoint from $S$ for almost every $v\in \RR^{d+j}$ with
$|v|<\delta$.
\end{proof}

\end{document}